\documentclass[10pt,reqno]{amsart}
\usepackage{amsfonts,latexsym,amssymb,amsmath,amsthm,mathtools}
\usepackage{amsrefs} 
\usepackage{enumerate}
\usepackage[all]{xy}
\usepackage{todonotes}
\usepackage{hyperref}

\makeatletter
\g@addto@macro\bfseries{\boldmath}
\makeatother

\hypersetup{
    colorlinks,
    linkcolor={red!50!black},
    citecolor={blue!50!black},
    urlcolor={blue!80!black}
}

\newcommand{\dV}{\,\mathrm{dV}}
\DeclareMathOperator{\curl}{curl}
\DeclareMathOperator{\Ric}{Ric}
\DeclareMathOperator{\tr}{tr}
\DeclareMathOperator{\vol}{vol}
\DeclareMathOperator{\OO}{O}
\DeclareMathOperator{\oo}{o}
\renewcommand{\div}{\mathrm{div}}
\newcommand{\C}{\mathbb{C}}
\newcommand{\Z}{\mathbb{Z}}
\newcommand{\R}{\mathbb{R}}
\newcommand{\K}{\mathcal{K}}
\renewcommand{\H}{\mathcal{H}}
\newcommand{\ddt}{\frac{d}{d t}}
\newcommand{\id}{\mathrm{id}}
\newcommand{\OC}[1]{\Omega_{C^\infty}^{#1}(M)}
\newcommand{\OL}[1]{\Omega_{L^2}^{#1}(M)}
\newcommand{\OD}[1]{\Omega_{\mathcal{D}'}^{#1}(M)}
\newcommand{\nh}{\hat\nabla}
\newcommand{\nb}{\bar\nabla}
\renewcommand{\sb}{\bar\ast}
\newcommand{\SOv}{\mathrm{SO}(4)}
\newcommand{\E}{\mathbf{E}}
\newcommand{\n}{\mathbf{n}}

\newcommand{\<}{\langle}
\renewcommand{\>}{\rangle}

 \newtheorem{theorem}{Theorem}[section]
 \newtheorem{lemma}[theorem]{Lemma}
 \newtheorem{corollary}[theorem]{Corollary}

 \theoremstyle{definition}
 \newtheorem{definition}[theorem]{Definition}
 \newtheorem{example}[theorem]{Example}
 \newtheorem{remark}[theorem]{Remark}

\newtagform{simple}{}{}{}

\allowdisplaybreaks
\title[The curl operator]{The curl operator on odd-dimensional manifolds}

\author[C.~B\"ar]{Christian B\"ar}
\address{Institut f\"ur Mathematik, Universit\"at Potsdam, Karl-Liebknecht-Str.~24-25, 14476 Potsdam, Germany}
\email{baer@math.uni-potsdam.de}

\date{\today}

\begin{document}

\begin{abstract}
We study the spectral properties of $\curl$, a linear differential operator of first order acting on differential forms of appropriate degree on an odd-dimensional closed oriented Riemannian manifold.
In three dimensions its eigenvalues are the electromagnetic oscillation frequencies in vacuum without external sources.
In general, the spectrum consists of the eigenvalue $0$ with infinite multiplicity and further real discrete eigenvalues of finite multiplicity.
We compute the Weyl asymptotics and study the $\zeta$-function.
We give a sharp lower eigenvalue bound for positively curved manifolds and analyze the equality case.
Finally, we compute the spectrum for flat tori, round spheres and $3$-dimensional spherical space forms.
\end{abstract}
\maketitle


\section*{I. Introduction}

Let $\Omega\subset\R^3$ be a domain.
The Maxwell equations in vacuum in absence of external sources are
\begin{align}
\curl E + \partial_t B &= 0, \label{eq:Max1}\\
\curl B - \partial_t E &= 0, \label{eq:Max2}\\
\div E &= 0, \label{eq:Max3}\\
\div B &= 0.
\label{eq:Max4}
\end{align}
Here $E$ and $B$ are time-dependent vector fields on $\Omega$, the electric and magnetic fields, respectively.
The equations have to be complemented with suitable boundary conditions.
The ansatz 
$$
E(t,x) = e^{i\lambda t}E_0(x), 
\quad
B(t,x) = e^{i\lambda t}B_0(x),
$$
yields a solution to the first two equations if and only if
$$
\begin{pmatrix}
0 & -i\curl \\
i\curl & 0
\end{pmatrix}
\begin{pmatrix}
E_0 \\ B_0
\end{pmatrix}
=
\lambda
\begin{pmatrix}
E_0 \\ B_0
\end{pmatrix}.
$$
Thus the eigenvalues of the ``stationary Maxwell operator'' $\begin{pmatrix} 0 & -i\curl \\ i\curl & 0\end{pmatrix}$ on divergence free vector fields are regarded as the electromagnetic oscillation frequencies of $\Omega$.
This spectrum has been studied by Weyl \cite{Weyl} on bounded domains with smooth boundary.
Weyl showed the asymptotic law
$$
N(\lambda) = \frac{\vol(\Omega)}{3\pi^2}\cdot\lambda^3 + \oo(\lambda^3)
$$
as $\lambda\to\infty$.
Here $N(\lambda)$ denotes the number of eigenvalues whose modulus is bounded from above by $\lambda$.
Safarov \cite{Saf} improved this to 
$$
N(\lambda) = \frac{\vol(\Omega)}{3\pi^2}\cdot\lambda^3 + \OO(\lambda^2)
$$
and, under an additional assumption on the billards of the domain, even to
$$
N(\lambda) = \frac{\vol(\Omega)}{3\pi^2}\cdot\lambda^3 + \oo(\lambda^2).
$$
The case when the boundary of $\Omega$ has only Lipschitz regularity has also been investigated, see e.g.\ \cite{BirSol1087,Fil,Ven2010}.
Additional complications arise for nonsmooth dielectric permittivity and magnetic permiability.
We will, however, consider only the case when they are constant and can be normalized to be $1$ by a suitable choice of physical units.

Maxwell's equations \eqref{eq:Max1}--\eqref{eq:Max4} make sense on any oriented Riemannian $3$-manifold~$M$.
If the manifold is compact and without boundary we need not worry about boundary conditions.
Then $\curl$ turns out to be a selfadjoint operator and $\lambda$ is an eigenvalue of $\curl$ if and only if $\lambda$ and $-\lambda$ are eigenvalues of the stationary Maxwell operator.

We will study the spectrum of $\curl$ on closed oriented Riemannian manifolds.
In order to generalize it to higher dimensions it is convenient to reformulate it in terms of differential forms rather than vector fields.
In three dimensions, $\curl$ can be equivalently defined acting on $1$-forms by $\curl=*d$ where $d$ denotes the exterior differential and $*$ the Hodge-star operator.

More generally, if the dimension $n$ of $M$ is odd the operator $*d$ acts on $\frac{n-1}{2}$-forms.
It turns out that $*d$ is formally selfadjoint if $n\equiv 3$ mod $4$ and formally skewadjoint if $n\equiv 1$ mod $4$.
To obtain a selfadjoint operator in all odd dimensions we define $\curl=i*d$ in the latter case.
Similar generalizations to higher dimensions using differential forms have been considered in the literature \cite{Weyl52, Mill, Weck, DF2008, Fil}.
In \cite[Thm~1.3]{JS} the connection between classical and quantum ergodicity of $\curl$ has been studied.
We hope that our investigation of the curl operator in higher dimensions may prove useful for the understanding of extensions of electromagnetism to compactified extra dimensions as well as the the $p$-brane scenario.

The present paper is structured as follows.
In Sec.~\ref{sec:2}, we fix notations and recall the Hodge decomposition theorem.
In Sec.~\ref{sec:3}, we introduce the $\curl$-operator on odd-dimensional oriented Riemannian manifolds and show essential selfadjointness if the manifold is closed.
The operator is not elliptic, indeed it has an infinite-dimensional kernel.
But the rest of the spectrum is discrete, i.e., consists of eigenvalues of finite multiplicity, and the corresponding eigenforms are smooth.
In dimension $3$, restricting to the complement of the kernel is equivalent to imposing equations \eqref{eq:Max3} and \eqref{eq:Max4}.

The structure of the spectrum is investigated in Sec.~\ref{sec:4}.
If $n\equiv 1$ mod $4$ the spectrum turns out to be symmetric about $0$ but for $n\equiv 3$ mod $4$ this is in general not the case.
We give an explicit example for $n=3$.
Denoting the number of positive eigenvalues below $\lambda$ by $N_+(\lambda)$ and that of negative eigenvalues above $-\lambda$ by $N_-(\lambda)$ we prove the Weyl asymptotics
$$
N_\pm(\curl,\lambda)
=
\frac{\vol(M)}{2\cdot\pi^{\frac{n+1}{2}}\cdot n\cdot \frac{n-1}{2}!}\cdot\lambda^{n} + \OO(\lambda^{n-1}).
$$
as $\lambda\to\infty$.
Then, we introduce the $\zeta$-function of $\curl$ and prove its basic properties.
In particular, the value at the origin $\zeta(0)$ turns out to be an integer-valued topological invariant of the underlying manifold.
When taken modulo two $\zeta(0)$ gives Kervaire's semi-characteristic of $M$.

Interestingly, the $\eta$-invariant of $\curl$ has been studied long ago by Millson.
In \cite{Mill} he shows that it coincides with the $\eta$-invariant of the signature operator acting on forms of even degree.
This $\eta$-invariant occurs as a boundary contribution in the signature formula for manifolds with boundary due to Atiyah, Singer, and Patodi \cite[Thm.~4.14]{APS}.

In Sec.~\ref{sec:5}, we prove a sharp lower eigenvalue estimate if the curvature operator of $M$ is positive.
In three dimensions this can be relaxed to a lower Ricci curvature bound.
The equality case is also analyzed.

In Sec.~\ref{sec:6}, we compute the $\curl$-spectrum for flat tori and round spheres.
In these cases the spectrum is always symmetric about $0$.
In dimension $3$ we also treat spherical space forms and obtain a convenient criterion for the symmetry of the spectrum.
Suitable lens spaces then provide simple examples for nonsymmetric $\curl$-spectrum.

\section*{II. Differential forms}
\label{sec:2}
\stepcounter{section}

We start by fixing some notations.
Throughout this text $M$ will denote an $n$-dimensional Riemannian manifold.
For $p\in \{0,1,\ldots,n\}$, we denote by $\OC{p}$, $\OL{p}$, and $\OD{p}$ the space of complex-valued $p$-forms on $M$ which are smooth, square-integrable, and distributional, respectively.
On $\OL{p}$ we have the scalar product
$$
(\omega_1,\omega_2) = \int_M \<\omega_1,\omega_2\> \dV
$$
turning $\OL{p}$ into a Hilbert space.
Here, $\<\cdot,\cdot\>$ denotes the scalar product on forms induced by the Riemannian metric and $\dV$ the Riemannian volume measure.

The exterior differential is denoted by $d:\OD{p}\to \OD{p+1}$.
Now assume that $M$ carries an orientation.
Then, the Hodge-star operator $\ast:\OD{p}\to \OD{n-p}$ is defined and characterized on $\OL{p}$ by
$$
(\omega_1,\omega_2) = \int_M \bar\omega_1\wedge\ast\omega_2 \, .
$$
The operator formally adjoint to $d$ is given by
\begin{equation}
d^\dagger = (-1)^{n(p+1)+1}\ast d\ast : \OD{p}\to \OD{p-1} \, 
\label{eq:dt}
\end{equation}
see e.g.\ \cite[p.~21]{Rosenberg}.
Moreover, we have on $\OD{p}$
\begin{align}
\ast^2 &= (-1)^{p(n-p)}\,\id\, , \label{eq:*2}  \\
\ast^\dagger &= (-1)^{p(n-p)}\ast \,  , \label{eq:*t} \\
d^2 &=(d^\dagger)^2 = 0 \, , 
\end{align}
see e.g.\ \cite[p.~33]{Besse}.
The Hodge-Laplacian is defined by
$$
\Delta = dd^\dagger + d^\dagger d : \OD{p}\to\OD{p}.
$$
It commutes with $d$, $d^\dagger$ and $\ast$.
If $M$ is closed, i.e., compact and without boundary, then there is the Hodge decomposition \cite[Ch.~6]{Warner}
\begin{equation}
\OD{p} = \ker(\Delta) \oplus d\OD{p-1} \oplus d^\dagger\OD{p+1} \, .
\label{eq:Hodge}
\end{equation}
Since the Hodge-Laplacian is elliptic its kernel is finite-dimensional and contained in $\OC{p}$.
Moreover, \eqref{eq:Hodge} and elliptic regularity theory imply
\begin{align}
\ker(d) &= \ker(\Delta) \oplus d\OD{p-1}, \label{eq:kerd} \\
\ker(d^\dagger) &= \ker(\Delta) \oplus d^\dagger\OD{p+1}. \label{eq:kerdt}
\end{align}

\section*{III. The curl operator}
\label{sec:3}
\stepcounter{section}

From now on, $M$ will always be oriented and of odd dimension $n$.
We consider the operator $\ast d:\OD{(n-1)/2}\to \OD{(n-1)/2}$.
Equivalently, it would also be possible to consider $d\ast :\OD{(n+1)/2}\to \OD{(n+1)/2}$ but we fix the other convention.

\subsection{Formal selfadjointness}

\begin{lemma}\label{lemma:symmetric}
Let $M$ be an oriented Riemannian manifold of odd dimension $n$.
Then, $\ast d:\OD{(n-1)/2}\to \OD{(n-1)/2}$ is formally selfadjoint if $n\equiv 3$ mod $4$ and formally skewadjoint if $n\equiv 1$ mod $4$.
\end{lemma}

\begin{proof}
By \eqref{eq:dt}--\eqref{eq:*t} and the fact that $n$ is odd we have
\begin{align*}
(\ast d)^\dagger 
&=
d^\dagger \ast^\dagger
=
(-1)^{n(n+3)/2+1} \ast d \ast (-1)^{(n-1)(n+1)/4}\ast \\
&=
(-1)^{n(n+3)/2+1} \ast d
=
(-1)^{(n+3)/2+1} \ast d \, .
\qedhere
\end{align*}
\end{proof}

In order to always have a formally selfadjoint operator we propose the following:

\begin{definition}
The operator
$$
\curl := 
\begin{cases}
i\ast d & \mbox{if }n\equiv 1 \mbox{ mod }4, \\
\ast d & \mbox{if }n\equiv 3 \mbox{ mod }4,
\end{cases}
$$
acting on $\OD{(n-1)/2}$ is called the \emph{curl operator}.
\end{definition}

For $n=1$ we locally have $M=\R$ and the curl operator is noting but $\curl = i\ddt$ acting on functions.
Therefore we will assume $n\ge 3$.
Then, $\ast d$ is not elliptic; 
in fact its kernel contains the infinite-dimensional space $d\OD{(n-3)/2}$.
In particular, eigenforms for the eigenvalue $0$ can have low regularity.

\begin{lemma}\label{lemma:curlDeltaEigen}
Let $M$ be an oriented closed Riemannian manifold of odd dimension~$n$.
Let $\omega\in\OD{(n-1)/2}$ and $\lambda\in\C\setminus\{0\}$.
Then, the following are equivalent:
\begin{enumerate}[(i)]
\item \label{curleigen}
$\omega=\omega_++\omega_-$ where $\curl\omega_+=\lambda\omega_+$ and $\curl\omega_-=-\lambda\omega_-$;
\item\label{Deltaeigen}
$\Delta\omega=\lambda^2\omega$ and $\omega$ is of the form $\omega=d^\dagger\eta$ for some $\eta\in\OD{(n+1)/2}$.
\end{enumerate}
\end{lemma}

\begin{proof}
To prove the implication \eqref{curleigen} $\Rightarrow$ \eqref{Deltaeigen} it suffices to consider $\curl\omega=\lambda\omega$.
Then, for $\tau= 1$ or $\tau = i$, we have $\lambda\omega=\tau\ast d \omega=\tau\ast d \ast\ast\omega=\pm\tau d^\dagger\ast\omega$.
Thus $\eta=\pm\lambda^{-1}\tau\ast\omega$ does the job.
Moreover,
$$
\Delta\omega 
=
(dd^\dagger + d^\dagger d)\omega
=
d^\dagger d\omega
=
(-1)^{(n+1)/2}\ast d\ast d \omega
=
\curl\curl\omega
=
\lambda^2\omega .
$$
Conversely, let $\omega=d^\dagger\eta$ satisfy $\Delta\omega=\lambda^2\omega$.
Then, the same computation shows $\curl\curl\omega=\lambda^2\omega$.
Since $\curl$ commutes with $\Delta$ it leaves its eigenspace $E(\Delta,\lambda^2)$ for the eigenvalue $\lambda^2$ invariant.
By \eqref{eq:dt} it also maps $d^\dagger\OD{(n+1)/2}$ to itself.
Thus $\curl$ restricts to an endomorphism on the finite-dimensional space $E(\Delta,\lambda^2)\cap d^\dagger\OD{(n+1)/2}$ whose square is $\lambda^2\cdot\id$.
This selfadjoint endomorphism can only have the eigenvalues $\lambda$ and $-\lambda$ and \eqref{curleigen} follows.
\end{proof}

For the eigenspace of any linear operator $L$ to the eigenvalue $\lambda$ we write $E(L,\lambda)$.
For the multiplicity we write $m(L,\lambda):=\dim(E(L,\lambda))$.

\begin{corollary}
Let $M$ be an oriented closed Riemannian manifold of odd dimension~$n$.
Then, eigenforms of $\curl$ to nonzero eigenvalues are smooth and the multiplicity of any nonzero eigenvalue is finite.
\end{corollary}

\begin{proof}
We may rewrite the statement of Lemma~\ref{lemma:curlDeltaEigen} as
\begin{equation}
E(\curl,\lambda) \oplus E(\curl,-\lambda) = E(\Delta,\lambda^2)\cap d^\dagger\OD{(n+1)/2} \, .
\label{eq:curlDeltaSpec}
\end{equation}
Since $E(\Delta,\lambda^2)$ is finite-dimensional and consists of smooth forms by elliptic theory the assertion follows.
\end{proof}

\begin{remark}
The proof of the implication \eqref{curleigen} $\Rightarrow$ \eqref{Deltaeigen} in Lemma~\ref{lemma:curlDeltaEigen} did not use the assumption that $M$ is closed.
Here $M$ might even be incomplete.
Thus smoothness of eigenforms of $\curl$ to nonzero eigenvalues is also true in this general case.
\end{remark}

\subsection{Selfadjointness}
By Lemma~\ref{lemma:symmetric} we know that $\curl$ defines a symmetric unbounded operator in the Hilbert space $\OL{(n-1)/2}$ with domain $\OC{(n-1)/2}$.

\begin{lemma}\label{lemma:selfadjoint}
Let $M$ be an oriented closed Riemannian manifold of odd dimension~$n$.
Then, $\curl$ with domain $\OC{(n-1)/2}$ is essentially selfadjoint in the Hilbert space $\OL{(n-1)/2}$.
\end{lemma}

\begin{proof}
It suffices to show that the adjoint operator $\curl^*$ (in the sense of functional analysis) of $\curl$ with domain $\OC{(n-1)/2}$ in $\OL{(n-1)/2}$ does not have nontrivial solutions of $\curl^*\omega=\pm i \omega$.
Then, $\omega\in\OL{(n-1)/2}$ is a distributional eigenform of $\curl$ to the eigenvalue $\pm i$.
By Lemma~\ref{lemma:curlDeltaEigen} $\omega$ is then an eigenform of $\Delta$ to the eigenvalue $-1$ and is, in particular, smooth.
Since $\Delta$ is nonnegative $\omega=0$.
\end{proof}

On $\R^7$ on can define a vector cross product and a corresponding curl operator based on the algebra of the octonions \cite{PY}.
Since this curl operator acts on vector fields while our curl in this case acts on $3$-forms which have fiber dimension $35$, there seems to be no relation.

\section*{IV. The spectrum}
\label{sec:4}
\stepcounter{section}

When we now speak of the spectrum of $\curl$ we mean the spectrum of its unique selfadjoint extension in $\OL{(n-1)/2}$.

\subsection{Structure of the spectrum}

\begin{theorem}
Let $M$ be an oriented closed Riemannian manifold of odd dimension~$n\ge3$.
Then, the continuous spectrum of $\curl$ is empty.
The point spectrum consists of the eigenvalue $0$ which has infinite multiplicity and the discrete spectrum.
\end{theorem}

\begin{proof}
By \eqref{eq:kerd} the kernel of $\curl$ is given by
$$
\ker(\curl) = \ker(d) = \ker(\Delta) \oplus \big(d\OD{(n-3)/2} \cap \OL{(n-1)/2}\big)
$$
where the second summand is obviously infinite-dimensional.
For the orthogonal complement Lemma~\ref{lemma:curlDeltaEigen} provides us with the spectral resolution
\begin{align*}
d^\dagger\OD{(n+1)/2} \cap \OL{(n-1)/2}
&=
\bigoplus_{\mu\in\sigma(\Delta)\setminus\{0\}}E(\Delta,\mu)\cap d^\dagger\OD{(n+1)/2} \\
&=
\bigoplus_{\mu\in\sigma(\Delta)\setminus\{0\}} E(\curl,\sqrt{\mu}) \oplus E(\curl,-\sqrt{\mu}) \, .
\end{align*}
Here $\sigma(\Delta)$ denotes the spectrum of the selfadjoint extension of $\Delta$ and the sum is a sum of Hilbert spaces in $\OL{(n-1)/2}$. 
Recall that $d^\dagger\OD{(n+1)/2}$ is left invariant by $\Delta$.
\end{proof}

\subsection{Symmetry of the spectrum}

Since $\curl$ has positive and negative eigenvalues the question arises whether the spectrum is symmetric about $0$.

\begin{theorem}\label{thm:SpecSymDim}
Let $M$ be an oriented closed Riemannian manifold of odd dimension~$n$ with $n\equiv 1$ mod $4$.
Then, the spectrum of $\curl$ is symmetric about $0$.
\end{theorem}

\begin{proof}
If $n\equiv 1$ mod $4$ then $*d=-i\curl$ restricts to a real skewsymmetric endomorphism on $E(\Delta,\mu)\cap d^\dagger\OD{(n+1)/2}$.
Thus on this subspace $-i\curl$ has the eigenvalues $i\sqrt{\mu}$ and $-i\sqrt{\mu}$ with equal multiplicity.
Hence, $\curl$ itself has the eigenvalues $\sqrt{\mu}$ and $-\sqrt{\mu}$ with equal multiplicity.
\end{proof}

In the last section we will exhibit a $3$-dimensional example with nonsymmetric spectrum.
But even when $n\equiv 3$ mod $4$ there are situations where the spectrum is necessarily symmetric.

\begin{theorem}\label{thm:SpecSymIso}
Let $M$ be an oriented closed Riemannian manifold of odd dimension~$n$.
Assume there exists an orientation reversing isometry $f:M\to M$.

Then, the spectrum of $\curl$ is symmetric about $0$.
\end{theorem}

\begin{proof}
The map $f$ acts by pull-back on $\OD{(n-1)/2}$ and commutes with $d$.
Since it is an orientation reversing isometry it anticommutes with the Hodge-star operator.
Hence, it anticommutes with $\curl$.
Thus $f^*$ restricts to an isomorphism $E(\curl,\lambda)\to E(\curl,-\lambda)$.
\end{proof}

\begin{corollary}
Let $M$ be an oriented closed Riemannian symmetric space of odd dimension~$n$.
Then, the spectrum of $\curl$ is symmetric about $0$.
\end{corollary}

\begin{proof}
Let $f$ be the geodesic reflection about a point in $M$.
Since $M$ is symmetric this is an isometry and since $n$ is odd $f$ is orientation reversing.
\end{proof}

Examples for such symmetric spaces are flat tori, round spheres, compact Lie groups with biinvariant metrics etc.

\subsection{Weyl asymptotics}
To examine the asymptotic behavior of large eigenvalues we introduce the eigenvalue counting functions and set for $\lambda>0$
$$
N_+(\curl,\lambda) := \sum_{0<\lambda'\leq\lambda}m(\curl,\lambda')
$$
and 
$$
N_-(\curl,\lambda) := \sum_{0<\lambda'\leq\lambda}m(\curl,-\lambda') \, .
$$
Hence, $N_+(\lambda)$ is the total number of positive eigenvalues below $\lambda$ and $N_-(\lambda)$ is the total number of negative eigenvalues above $-\lambda$.
Similarly, we have the counting functions for the Hodge-Laplacians
$$
N(p,\lambda) := \sum_{0<\lambda'\leq\lambda}m(\Delta|_{\OL{p}},\lambda') \, .
$$

\begin{lemma}\label{lem:N++N-}
Let $M$ be an oriented closed Riemannian manifold of odd dimension~$n$.
Then,
$$
N_+(\curl,\lambda)+N_-(\curl,\lambda)
=
(-1)^{\frac{n-1}{2}}\sum_{p=0}^{\frac{n-1}{2}}(-1)^p N(p,\lambda^2) .
$$
\end{lemma}

\begin{proof}
The commutative diagram 
$$
\xymatrix{
d^\dagger\OD{p+1}\ar[d]^\Delta \ar[r]^d_\cong & d\OD{p} \ar[d]^\Delta \\
d^\dagger\OD{p+1} \ar[r]^d_\cong & d\OD{p}  
}
$$
shows that for fixed $\lambda>0$
\begin{align*}
m(\Delta|_{d^\dagger\OD{(n+1)/2}},\lambda^2) 
&=
m(\Delta|_{\OD{(n-1)/2}},\lambda^2) - m(\Delta|_{d\OD{(n-3)/2}},\lambda^2) \\
&=
m(\Delta|_{\OD{(n-1)/2}},\lambda^2) - m(\Delta|_{d^\dagger\OD{(n-1)/2}},\lambda^2) \, .
\end{align*}
Proceeding inductively we get
$$
m(\Delta|_{d^\dagger\OD{(n+1)/2}},\lambda^2)
=
(-1)^{\frac{n-1}{2}}\sum_{p=0}^{\frac{n-1}{2}}(-1)^p\,  m(\Delta|_{\OD{p}},\lambda^2)\, 
$$
and hence, by \eqref{eq:curlDeltaSpec},
\begin{align}
m(\curl,\lambda) + m(\curl,-\lambda)  
&= (-1)^{\frac{n-1}{2}}\sum_{p=0}^{\frac{n-1}{2}}(-1)^p\,  m(\Delta|_{\OD{p}},\lambda^2)\, .
\label{eq:mlambda}
\end{align}

Summation over $\lambda$ proves the assertion.
\end{proof}

\begin{theorem}\label{thm:Weyl}
Let $M$ be an oriented closed Riemannian manifold of odd dimension~$n$.
Then, as $\lambda\to\infty$, 
$$
N_\pm(\curl,\lambda)
=
\frac{\vol(M)}{2\cdot\pi^{\frac{n+1}{2}}\cdot n\cdot \frac{n-1}{2}!}\cdot\lambda^{n} + \OO(\lambda^{n-1}).
$$
\end{theorem}

\begin{proof}
We apply \cite[Thm.~0.1]{Ivrii} to $A=\curl$ and the subspace $H=d^\dagger\OD{\frac{n+1}{2}}\cap \OL{\frac{n-1}{2}}$ of $\OL{\frac{n-1}{2}}$.
In other words, $H$ is the $L^2$-orthogonal complement of the kernel of $\curl$.
Then, we get $N_\pm(\curl,\lambda)=\kappa_\pm \lambda^n + \OO(\lambda^{n-1})$ where 
\begin{equation}
\kappa_\pm= (2\pi)^{-n}\int_{T^*M}\tr (\hat\pi_{\pm}(\xi)\pi(\xi))dxd\xi.
\label{eq:kappapm}
\end{equation}
Here $\pi(\xi)$ is the orthoprojection onto the orthogonal complement of $\xi\wedge\Lambda^{\frac{n-3}{2}}T^*M$ in $\Lambda^{\frac{n-1}{2}}T^*_xM$ and $\hat\pi_+(\xi) = \hat\pi(\xi,1)-\hat\pi(\xi,0)$ as well as $\hat\pi_-(\xi) = \hat\pi(\xi,0)-\hat\pi(\xi,1)$ where $\hat\pi(\xi,\lambda)$ is the spectral resolution of the principal symbol of $\curl$.
Since $\curl$ is a differential operator of first order its principal symbol depends linearly on $\xi$ and hence $\hat\pi_\pm(-\xi)=\hat\pi_\mp(\xi)$.
This implies
$$
\tr (\hat\pi_{\pm}(-\xi)\pi(-\xi))
=
\tr (\hat\pi_{\mp}(\xi)\pi(\xi))
$$
and therefore $\kappa_+=\kappa_-$.

It remains to determine this coefficient.
It is known (see e.g.\ \cite[Cor.~2.43]{BGV}) that $N(p,\lambda)$ has the following asymptotics as $\lambda\to\infty$:
\begin{equation}
N(p,\lambda) \sim \frac{{n \choose p}\cdot\vol(M)}{(4\pi)^{n/2}\cdot\Gamma(\frac{n}{2}+1)} \lambda^{n/2}.
\label{eq:Weyl}
\end{equation}
Inserting this into Lemma~\ref{lem:N++N-} yields
\begin{align*}
N_+(\curl,\lambda)+N_-(\curl,\lambda)
&\sim
(-1)^{\frac{n-1}{2}}\sum_{p=0}^{\frac{n-1}{2}}(-1)^p  \frac{{n \choose p}\cdot\vol(M)}{(4\pi)^{n/2}\cdot\Gamma(\frac{n}{2}+1)} \lambda^{n} \\
&=
\frac{{n-1 \choose \frac{n-1}{2}}\cdot\vol(M)}{(4\pi)^{n/2}\cdot\Gamma(\frac{n}{2}+1)} \lambda^{n}.
\end{align*}
Here we employed the formula
\begin{equation*}
\sum_{p=0}^k (-1)^p {2k+1 \choose p} = (-1)^k {2k \choose k}
\end{equation*}
with $k=\frac{n-1}{2}$.
Using Legendre's duplication formula for the $\Gamma$-function
$$
\Gamma\Big(\frac{n+1}{2}\Big)\Gamma\Big(\frac{n+2}{2}\Big) = 2^{-n}\cdot\sqrt{\pi}\cdot\Gamma(n+1)
$$
we obtain for the dimension-dependent coefficient
\begin{align*}
\frac{{n-1 \choose \frac{n-1}{2}}}{(4\pi)^{n/2}\cdot\Gamma(\frac{n}{2}+1)} 
=&\,\,
\frac{{n-1 \choose \frac{n-1}{2}}\cdot\Gamma(\frac{n+1}{2})}{\pi^{(n+1)/2}\cdot\Gamma(n+1)} \\
=&\,\,
\frac{\frac{(n-1)!}{(\frac{n-1}{2}!)^2}\cdot\frac{n-1}{2}!}{\pi^{(n+1)/2}\cdot n!} \\
=&\,\,
\frac{1}{\pi^{(n+1)/2}\cdot n\cdot \frac{n-1}{2}!}.
\end{align*}
This shows
$$
2\kappa_\pm 
=
\kappa_++\kappa_-
=
\frac{\vol(M)}{\pi^{(n+1)/2}\cdot n\cdot \frac{n-1}{2}!}
$$
and concludes the proof.
\end{proof}

\begin{remark}
For low dimensions $n$ one can compute the coefficient of the leading term in the Weyl expansion directly from \eqref{eq:kappapm}.
For $n=1$ we have $\pi(\xi)=\id$ and the principal symbol of $\curl$ at $\xi$ is multiplication with $\pm|\xi|$ where the sign depends on the orientation of $\xi$.
Thus $\hat\pi_\pm(\xi)=\id$ if $0<|\xi|\leq1$ and $\xi$ is correctly oriented and $\hat\pi_\pm(\xi)=0$ otherwise.
Hence, for fixed $x$ the integral over $T^*_xM$ gives $1$.
Therefore $\kappa_\pm= \frac{\vol(M)}{2\pi}$ which coincides with the coefficient in Theorem~\ref{thm:Weyl}.

For $n=3$, $\pi(\xi)$ is the orthoprojection onto the orthogonal complement of $\xi$ in $T^*_xM$.
The principal symbol of $\curl$ is $i|\xi|$ times a rotation in the plane $\xi^\perp$ and hence has the eigenvalues $|\xi|$ and $-|\xi|$.
Thus $\tr (\hat\pi_{\pm}(\xi)\pi(\xi))=1$ if $|\xi|\leq 1$ and vanishes otherwise.
Therefore the integral over $T^*M$ coincides with the volume of the unit ball.
Hence
$$
\kappa_\pm
=
(2\pi)^{-3}\cdot\frac{4\pi}{3}\cdot\vol(M)
=
\frac{\vol(M)}{6\pi^2},
$$
again in accordance with Theorem~\ref{thm:Weyl}.
This is also consistent with the formulas obtained in \cite{Weyl,Saf} for domains in $\R^3$.
\end{remark}

\subsection{The ${\zeta}$-function}
We define the $\zeta$-function of $\curl$ by
\begin{align*}
\zeta(s) &= 
\sum_{\lambda\neq 0} m(\curl,\lambda) \cdot|\lambda|^{-s}.
\end{align*}

\begin{theorem}\label{thm:zeta0}
The $\zeta$-function converges and is holomorphic for $\mathrm{Re}(s)>n$ and has a meromorphic continuation to $\C$.
The poles are simple and can occur only at $s=n,\, n-2,\ldots,1$.
Moreover,
$$
\zeta(0) = (-1)^{\frac{n+1}{2}}\sum_{p=0}^{\frac{n-1}{2}}(-1)^p\, b_p(M)
$$
where $b_p(M)$ denotes the $p^\mathrm{th}$ Betti number of $M$.
\end{theorem}

\begin{proof}
The $\zeta$-function of $\curl$ relates to the $\zeta$-functions of the Hodge-Laplacians
$$
\zeta(p,s) = \sum_{\lambda>0} m(\Delta|_{\OD{p}},\lambda)\cdot\lambda^{-s}.
$$
Namely, by \eqref{eq:mlambda} we find
$$
\zeta(s) 
= 
(-1)^{\frac{n-1}{2}}\sum_{p=0}^{\frac{n-1}{2}}(-1)^p\, \zeta(p,s/2)\, .
$$
The assertions about convergence, meromorphic continuation and the poles now follow directly from the corresponding statements for $\zeta(p,s)$, see e.g.\ \cite[Thm.~5.2]{Rosenberg}.
Moreover, again by \cite[Thm.~5.2]{Rosenberg} and by Hodge theory, we find
\begin{align*}
\zeta(0)
&=
(-1)^{\frac{n-1}{2}}\sum_{p=0}^{\frac{n-1}{2}}(-1)^p\, \zeta(p,0) \\
&=
(-1)^{\frac{n+1}{2}}\sum_{p=0}^{\frac{n-1}{2}}(-1)^p\, \dim\ker(\Delta|_{\OD{p}})\\
&=
(-1)^{\frac{n+1}{2}}\sum_{p=0}^{\frac{n-1}{2}}(-1)^p\, b_p(M).
\qedhere
\end{align*}
\end{proof}

In particular, the value $\zeta(0)$ is a topological invariant of $M$.
When taken modulo $2$ it is known as the \emph{semi-characteristic} of $M$ \cite{Ker}.

\subsection{The ${\eta}$-invariant}
An interesting modification of the $\zeta$-function is the $\eta$-function given by
$$
\eta(s) = \sum_{\lambda>0} \big(m(\curl,\lambda)-m(\curl,-\lambda) \big)\cdot\lambda^{-s}.
$$
Millson showed in \cite{Mill} that the $\eta$-invariant $\eta(0)$ coincides with the $\eta$-invariant of the signature operator acting on forms of even degree.
This $\eta$-invariant occurs as a boundary contribution in the signature formula for manifolds with boundary due to Atiyah, Singer, and Patodi \cite[Thm.~4.14]{APS}.

\section*{V. Eigenvalue estimates}
\label{sec:5}
\stepcounter{section}

In section~\ref{sec:Examples} we will compute the spectrum of $\curl$ on some particularly nice spaces.
In general, an explicit computation is not possible. 
But often one can at least give bounds on the spectrum.

To formulate an estimate which is valid in all odd dimensions consider the \emph{curvature operator} $\K$, a field of symmetric endomorphisms of  $\Lambda^2TM$.
It is characterized by 
$$
\<\K(X\wedge Y),U\wedge V\> =  \<R(X,Y)V,U\>
$$
for all $X,Y,U,V\in T_xM$ and all $x\in M$.
The manifold $M$ has constant sectional curvature $\kappa$ if and only if $\K=\kappa\cdot\id$.

\begin{theorem}\label{thm:AllDest}
Let $M$ be an oriented closed Riemannian manifold of odd dimension $n\ge3$.
Let $\kappa$ be a positive constant and assume $\K\ge\kappa\cdot\id$.
Then, all nonzero eigenvalues $\lambda$ of $\curl$ satisfy
$$
|\lambda| \ge \frac{n+1}{2}\sqrt{\kappa}.
$$
\end{theorem}

\begin{proof}
By \cite[Thm.~6.13]{GaMe} all eigenvalues $\mu$ of the Hodge-Laplacian on coexact $\frac{n-1}{2}$-forms satisfy
$$
\mu \ge \Big(\frac{n+1}{2}\Big)^2\kappa.
$$
Lemma~\ref{lemma:curlDeltaEigen} yields the claim.
\end{proof}

The estimate is sharp because equality holds for the standard sphere, see Theorem~\ref{thm:roundsphere} below.
Unfortunately, positivity of the curvature operator is a very strong assumption.
In dimension $3$ we now replace it by a weaker Ricci curvature bound.
The conclusion remains the same.

\begin{theorem}\label{thm:3Dest}
Let $M$ be an oriented closed $3$-dimensional Riemannian manifold.
Let $\kappa$ be a positive constant and assume $\Ric\ge2\kappa\cdot\id$.
Then, all nonzero eigenvalues $\lambda$ of $\curl$ satisfy
$$
|\lambda| \ge 2\sqrt{\kappa}.
$$
\end{theorem}

Again, the estimate is optimal because equality is attained on the round $S^3$.

\begin{proof}
We introduce an auxiliary connection on $T^*M$ by
$$
\nh_X\omega := \nabla_X\omega +\sqrt{\kappa} *(X^\flat \wedge\omega).
$$
Here $X^\flat$ denotes the covector corresponding to $X$ under the ``musical isomorphism'', i.e., $X^\flat(Y)=\<X,Y\>$ for all vectors $Y$.
This defines a metric connection $\nh$ because the term we have added is skewsymmetric in $\omega$.

We compute the connection-Laplacian for $\nh$.
We fix a point $x$ in $M$ and choose a local orthonormal tangent frame $e_1,e_2,e_3$ near $x$ which is synchronous at $x$, i.e., $\nabla_\cdot e_j=0$ at $x$.
Then, we find at $x$:
\begin{align*}
\nh^*\nh\omega
&=
-\sum_{j=1}^3 \nh_{e_j}\nh_{e_j}\omega\\
&=
-\sum_{j=1}^3 \big(\nabla_{e_j}\nabla_{e_j}\omega +2\sqrt{\kappa}*(e_j^\flat\wedge\nabla_{e_j}\omega)+ \kappa*(e_j^\flat\wedge *(e_j^\flat\wedge\omega)) \big)\\
&=
\nabla^*\nabla\omega - 2\sqrt{\kappa}\curl\omega +2\kappa\omega.
\end{align*}
Inserting the Bochner formula
$$
\Delta = \nabla^*\nabla + \Ric
$$
yields
\begin{equation}
\Delta-2\sqrt{\kappa}\curl = \nh^*\nh + \Ric - 2\kappa.
\label{eq:Weitzen}
\end{equation}
Now let $\lambda>0$ be an eigenvalue of $\curl$ with corresponding eigenform $\omega$.
Inserting $\omega$ into \eqref{eq:Weitzen} and taking the $L^2$-scalar product with $\omega$ yields
\begin{align*}
(\lambda^2 - 2\sqrt{\kappa}\lambda)\|\omega\|^2
&=
\|\nh\omega\|^2 + (\Ric(\omega),\omega) -2\kappa\|\omega\|^2\\
&\ge
0+2\kappa\|\omega\|^2-2\kappa\|\omega\|^2\\
&= 0.
\end{align*}
Hence, $(\lambda-2\sqrt{\kappa})\lambda\ge0$ and, since $\lambda>0$, we conclude $\lambda\ge2\sqrt{\kappa}$.

For a negative eigenvalue $\lambda$ we can obtain the estimate using the connection $\check\nabla_X\omega := \nabla_X\omega -\sqrt{\kappa} *(X^\flat \wedge\omega)$ or, alternatively, we reduce to the case of positive $\lambda$ by reversing the orientiation.
\end{proof}

\begin{remark}
It is possible to deduce Theorem~\ref{thm:AllDest} in a similar fashion using the modified connection
$$
\nh_X\omega := \nabla_X\omega +\alpha\sqrt{\kappa} *(X^\flat \wedge\omega)
$$
where the optimal value of 
$$
\alpha\in
\begin{cases}
\R, & \mbox{ if }n\equiv3 \mbox{ mod }4,\\
i\R, & \mbox{ if }n\equiv1 \mbox{ mod }4,
\end{cases}
$$
depends on the dimension.
\end{remark}

It is interesting to compare the estimate in Theorem~\ref{thm:3Dest} to Lichnerowicz' lower bound (see e.g.\ \cite[p.~82]{Chavel}) for the first eigenvalue $\mu$ of the Laplacian acting on functions (under the same Ricci curvature assumption):
\begin{equation}
\mu \ge 3\kappa.
\label{eq:Lichnerowicz}
\end{equation}
If equality holds in \eqref{eq:Lichnerowicz} then Obata's theorem tells us that $M$ is isometric to a round sphere.
We have a similar rigidity statement for Theorem~\ref{thm:3Dest} as well.
On the round $3$-sphere the multiplicity of the eigenvalue $\lambda=2$ is $3$.
Conversely, we can now show:

\begin{theorem}\label{thm:Rigidity3D}
Let $M$ be an oriented closed and connected $3$-dimensional Riemannian manifold.
Let $\kappa$ be a positive constant and assume $\Ric\ge2\kappa\cdot\id$.
Assume that $\lambda = 2\sqrt{\kappa}$ or $\lambda = -2\sqrt{\kappa}$ is an eigenvalue of $\curl$ of multiplicity at least $2$.

Then, $M$ has constant sectional curvature $\kappa$ and is hence a spherical spaceform.
Moreover, if both $2\sqrt{\kappa}$ and $-2\sqrt{\kappa}$ are $\curl$-eigenvalues of multiplicity $2$ at least, then $M$ is isometric to $S^3$ or to $\mathbb{RP}^3$ equipped with a metric of constant sectional curvature $\kappa$.
\end{theorem}

\begin{proof}
By reversing the orientation if necessary we can assume that $\lambda$ is positive.
By rescaling the metric we may furthermore assume that $\kappa=1$.

Thus let $\lambda = 2$ be an eigenvalue of $\curl$ of multiplicity at least $2$.
Every eigenform $\omega$ of $\curl$ to the eigenvalue $\lambda$ must be parallel with respect to the connection $\nh_X\omega = \nabla_X\omega + *(X^\flat \wedge\omega)$, see the proof of Theorem~\ref{thm:3Dest}.
Since the connection $\nh$ is metric we can choose the $\omega_j$ such that they are perpendicular and of length $1$ at each point.
One easily checks that $\omega_3:=*(\omega_1\wedge\omega_2)$ is also $\nh$-parallel and complements $\omega_1$ and $\omega_2$ to an orthonormal basis at each point.
Thus the cotangent bundle $T^*M$ is trivialized by the $\nh$-parallel forms $\omega_1$, $\omega_2$ and $\omega_3$.

Let $V_j=\omega_j^\sharp$ be the corresponding vector fields.
Without loss of generality we assume that $V_1,V_2,V_3$ is positively oriented.
Since the $\omega_j$ are $\nh$-parallel we have
$$
\nabla_{V_i}\omega_j 
= -*(\omega_i\wedge\omega_j)
=
\begin{cases}
0,& \mbox{ if }i=j,\\
-\omega_k, & \mbox{ if }\{i,j,k\}=\{1,2,3\}\mbox{ and }(i,j,k)\mbox{ is even},\\
\omega_k, & \mbox{ if }\{i,j,k\}=\{1,2,3\}\mbox{ and }(i,j,k)\mbox{ is odd}.
\end{cases}
$$
This implies 
$$
\nabla_{V_i}V_j 
=
\begin{cases}
0,& \mbox{ if }i=j,\\
-V_k, & \mbox{ if }\{i,j,k\}=\{1,2,3\}\mbox{ and }(i,j,k)\mbox{ is even},\\
V_k, & \mbox{ if }\{i,j,k\}=\{1,2,3\}\mbox{ and }(i,j,k)\mbox{ is odd}.
\end{cases}
$$
Hence, if $i,j,k$ are pairwise disjoint we get 
\begin{align*}
R(V_i,V_j)V_k
&=
\nabla_{V_i}\nabla_{V_j}V_k - \nabla_{V_j}\nabla_{V_i}V_k - \nabla_{\nabla_{V_i}V_j}V_k + \nabla_{\nabla_{V_j}V_i}V_k\\
&=
\pm \nabla_{V_k}V_k \pm \nabla_{V_k}V_k \pm \nabla_{V_k}V_k\pm \nabla_{V_k}V_k\\
&=0.
\end{align*}
If the permutation $(i,j,k)$ is even we find
\begin{align*}
R(V_i,V_j)V_j
&=
\nabla_{V_i}\nabla_{V_j}V_j - \nabla_{V_j}\nabla_{V_i}V_j - \nabla_{\nabla_{V_i}V_j}V_j + \nabla_{\nabla_{V_j}V_i}V_j\\
&=
0+\nabla_{V_j}V_k+\nabla_{V_k}V_j+\nabla_{V_k}V_j\\
&=
V_i
\end{align*}
and the same result also holds if $(i,j,k)$ is odd.
This determines the full curvature tensor which must then be given by
$$
R(X,Y)Z = \<Y,Z\>X-\<X,Z\>Y.
$$
Thus $M$ has constant sectional curvature $1$.

Now assume that both $2$ and $-2$ are $\curl$-eigenvalues of multiplicity $2$ at least.
Then, as we have seen above, they actually have multiplicity $3$.
The assertion will be shown right after Corollary~\ref{cor:SphericalKleinst} below.
\end{proof}

\begin{remark}
Theorems~\ref{thm:3Dest} and the first part of \ref{thm:Rigidity3D} can also be derived from Theorem~7.6 in \cite{CT94}.
Namely, by this result eigen-1-forms to the eigenvalues $\pm2\sqrt{\kappa}$ are dual to Killing vector fields which are pointwise eigenvectors of the Ricci curvature tensor.
\end{remark}

\section*{VI. Examples}\label{sec:Examples}
\label{sec:6}
\stepcounter{section}

We now consider a few examples of manifolds on which the spectrum of $\curl$ can be computed explicitly.
The equivariant $\eta$-invariant of $\curl$ for these spaces has been computed with representation theoretic methods by Millson in \cite{Mill}.

\subsection{Flat tori}
Let $\Gamma\subset\R^n$ be a lattice and $\Gamma^*\subset\R^n$ its dual lattice,
$$
\Gamma^* = \{\gamma\in\R^n \mid \<\gamma,\mu\>\in\Z \mbox{ for all }\mu\in\Gamma\}.
$$
We determine the eigenvalues of $\curl$ on the flat torus $M=\R^n/\Gamma$.
By Theorem~\ref{thm:SpecSymIso} $m(\curl,\lambda)=m(\curl,-\lambda)$.
Hence, \eqref{eq:mlambda} yields
\begin{equation}
m(\curl,\lambda)  
= 
(-1)^{\frac{n-1}{2}}\frac12\sum_{p=0}^{\frac{n-1}{2}}(-1)^p\,  m(\Delta|_{\OD{p}},\lambda^2)\, .
\label{eq:torus1}
\end{equation}
On a flat torus we have $m(\Delta|_{\OD{p}},\lambda^2)={n \choose p}m(\Delta|_{\OD{0}},\lambda^2)$.
Inserting this into \eqref{eq:torus1} yields
\begin{align*}
m(\curl,\lambda)  
&= 
(-1)^{\frac{n-1}{2}}\frac12\sum_{p=0}^{\frac{n-1}{2}}(-1)^p\,  {n \choose p} m(\Delta|_{\OD{0}},\lambda^2) \\
&=
\frac12{n-1 \choose \frac{n-1}{2}}m(\Delta|_{\OD{0}},\lambda^2) \, .
\end{align*}
The spectrum of the Laplace-Beltrami operator on a flat torus can be computed using Fourier series and is well known to be
$$
m(\Delta|_{\OD{0}},\lambda^2) 
=
\#\bigg\{\mu\in\Gamma^* \,\bigg|\, |\mu|=\frac{|\lambda|}{2\pi}\bigg\} \, ,
$$
see \cite[Prop.~B.I.2]{BGM}.
We summarize:

\usetagform{simple}
\begin{theorem}
On the flat torus $M=\R^n/\Gamma$ a number $\lambda\neq0$ is an eigenvalue of the operator $\curl$ if and only if there exists a $\mu\in\Gamma^*$ such that $|\lambda|=2\pi|\mu|$.
The multiplicity of $\lambda$ then is 
\begin{equation*}
m(\curl,\lambda) 
=
\frac12{n-1 \choose \frac{n-1}{2}}\cdot \#\bigg\{\mu\in\Gamma^* \,\bigg|\, |\mu|=\frac{|\lambda|}{2\pi}\bigg\} \, .
\tag{\qed}
\end{equation*}
\end{theorem}
\usetagform{default}

\subsection{Round spheres}
Now let $M=S^n$ be the round sphere with constant sectional curvature $1$.
Again by Theorem~\ref{thm:SpecSymIso} the spectrum of $\curl$ is symmetric about $0$.
Theorem~6 in \cite{IK} tells us that $\lambda^2$ is an eigenvalue of the Hodge-Laplacian on $d^\dagger\OD{(n+1)/2}$ if and only if it is of the form $\lambda^2=(\frac{n+1}{2}+k)^2$ for $k=0,1,2,\ldots$.
By \eqref{eq:curlDeltaSpec} and \cite[Thm.~6]{IK} the multiplicity is then given by
\begin{align*}
2m(\curl,\lambda)
&=
m(\Delta|_{d^\dagger\OD{(n+1)/2}},\lambda^2) \\
&=
m(\Delta|_{\OD{(n+1)/2}\cap\ker(d)},\lambda^2) \\
&=
\frac{(n+k)!\cdot(n+2k+1)}{\frac{n-1}{2}!\cdot k!\cdot\frac{n-1}{2}!\cdot(\frac{n+1}{2}+k)\cdot(\frac{n+1}{2}+k)} \\
&=
2\frac{(n+k)!}{(\frac{n-1}{2}!)^2\cdot k!\cdot(\frac{n+1}{2}+k)}.
\end{align*}
We summarize:
\usetagform{simple}
\begin{theorem}\label{thm:roundsphere}
On the round sphere $M=S^n$ with sectional curvature $1$ a number $\lambda\neq0$ is an eigenvalue of the operator $\curl$ if and only if it is of the form 
$$
\lambda=\pm\Big(\frac{n+1}{2}+k\Big)
$$
for some $k=0,1,2,\ldots$.
The multiplicity of $\lambda$ then is
\begin{equation*}
m(\curl,\lambda)
=
\frac{(n+k)!}{(\frac{n-1}{2}!)^2\cdot k!\cdot(\frac{n+1}{2}+k)}.
\tag{\qed}
\end{equation*}
\end{theorem}
\usetagform{default}

\begin{remark}
It is interesting to compare the spectrum of $\curl$ on $S^n$ to that of the Dirac operator acting on spinor fields.
By \cite[Thm.~1]{B1} the Dirac eigenvalues are the numbers given by
\begin{equation}
\pm\Big(\frac{n}{2}+k\Big), \quad k=0,1,2,\ldots
\label{eq:DiracS1}
\end{equation}
There now seems to be a contradiction for $n=1$ because $\curl$ then reduces to the Dirac operator.
The point is here that $S^1$ carries two different spin structures.
While \eqref{eq:DiracS1} gives the Dirac spectrum for the ``nontrivial'' spin structure, the formula in Theorem~\ref{thm:roundsphere} provides it for the ``trivial'' spin structure.
\end{remark}

\subsection{Spherical space forms}\label{subsec:SpaceForms}

We now study spherical space forms in $3$ dimensions, in other words, quotients of the round $3$-sphere $S^3$.
The group of orientation preserving isometries of $S^3$ is $\SOv$ acting by matrix multiplication from the left.
Oriented compact connected $3$-manifolds of constant sectional curvature $1$ are of the form $M=\Gamma\backslash S^3$ where $\Gamma\subset\SOv$ is a finite fixed point free subgroup.
One-forms on $M$ correspond to $\Gamma$-invariant one-forms on $S^3$ via pull-back along the projection map $S^3 \to \Gamma\backslash S^3$.
Hence, $M$ has the same $\curl$-eigenvalues as $S^3$, only the multiplicities on $M$ will in general be smaller than on $S^3$ (including the possibility $0$).
We encode this information in the following \emph{Poincar\'e series}:
\begin{align*}
F_+(z) &= \sum_{k=0}^\infty m(\curl,2+k)z^k,\\
F_-(z) &= \sum_{k=0}^\infty m(\curl,-(2+k))z^k,
\end{align*}
where the multiplicities $m(\curl,\pm(2+k))$ are those of $\curl$ on $M$.
Knowing the $\curl$-spectrum on $M$ is equivalent to knowing the power series $F_+(z)$ and $F_-(z)$.

\begin{lemma}
The power series $F_+(z)$ and $F_-(z)$ converge absolutely for $|z|<1$.
\end{lemma}

\begin{proof}
By Theorem~\ref{thm:roundsphere} with $n=3$ both $F_+(z)$ and $F_-(z)$ can be majorized by 
$$
\sum_{k=0}^\infty (k+3)(k+1)z^k .
$$
That power series has convergence radius $1$ because
\begin{equation*}
\lim_{k\to\infty} \frac{(k+3)(k+1)}{(k+4)(k+2)} = 1.
\qedhere
\end{equation*}
\end{proof}

The $\SOv$-module of $2$-forms decomposes into those of selfdual and antiselfdual $2$-forms, $\Lambda^2\R^4 = \Lambda^+ \oplus \Lambda^-$.
Let $\chi^\pm:\SOv\to\R$ be the corresponding characters.

\begin{theorem}\label{thm:SpecSpaceform}
Let $\Gamma\subset\SOv$ be a finite fixed point free subgroup.
Then, the $\curl$-spectrum of $M=\Gamma\backslash S^3$ is given by
\begin{align*}
F_+(z) &= \frac{1}{1+z^2}\Big(1+\frac{1}{|\Gamma|}\sum_{\gamma\in\Gamma} \frac{\chi^+(\gamma)-1 -z^2(\chi^-(\gamma)-1)}{\det(1-z\gamma)}\Big),\\
F_-(z) &= \frac{1}{1+z^2}\Big(1+\frac{1}{|\Gamma|}\sum_{\gamma\in\Gamma} \frac{\chi^-(\gamma)-1 -z^2(\chi^+(\gamma)-1)}{\det(1-z\gamma)}\Big).
\end{align*}
\end{theorem}

\begin{example}
We determine the $\curl$-spectrum of real projective $3$-space $\mathbb{RP}^3$.
In this case $\Gamma=\{1,-1\}$ and both $\gamma=1$ and $\gamma=-1$ act trivially on $2$-forms.
Hence, $\chi^+(\gamma)=\chi^-(\gamma)=3$.
Therefore
\begin{align*}
F_\pm(z) 
&=
\frac{1}{1+z^2}\Big(1+\frac{1}{2}\Big\{\frac{3-1-z^2(3-1)}{(1-z)^4}+\frac{3-1-z^2(3-1)}{(1+z)^4}\Big\}\Big)\\
&=
\frac{1}{1+z^2}\Big(1+\frac{1+z}{(1-z)^3}+\frac{1-z}{(1+z)^3}\Big)\\
&=
\sum_{j=0}^\infty (4(j+1)^2-1)z^{2j} \, .
\end{align*}
This shows that on $\mathbb{RP}^3$ the number $\pm(2+k)$ is a $\curl$-eigenvalue if and only if $k$ is even and in this case it has the same multiplicity as on $S^3$.
\end{example}

For the smallest positive and the largest negative $\curl$-eigenvalue of a spherical space form we get

\begin{corollary}\label{cor:SphericalKleinst}
The multiplicity of the smallest positive $\curl$-eigenvalue of $M=\Gamma\backslash S^3$ is given by
$$
m(\curl,2) = \frac{1}{|\Gamma|}\sum_{\gamma\in\Gamma}\chi^+(\gamma) \le 3.
$$
The maximal value $3$ is attained if and only if $\Gamma$ acts trivially on $\Lambda^+$.
Similarly, the multiplicity of the largest negative $\curl$-eigenvalue is given by
$$
m(\curl,-2) = \frac{1}{|\Gamma|}\sum_{\gamma\in\Gamma}\chi^-(\gamma) \le 3.
$$
The maximal value $3$ is attained if and only if $\Gamma$ acts trivially on $\Lambda^-$.
\end{corollary}

\begin{proof}
The multiplicity is given by
\begin{equation*}
m(\curl,\pm 2)
=
F_\pm(0)
=
1+\frac{1}{|\Gamma|}\sum_{\gamma\in\Gamma}\frac{\chi^\pm(\gamma) -1}{1}
=
\frac{1}{|\Gamma|}\sum_{\gamma\in\Gamma}\chi^\pm(\gamma).
\end{equation*}
Now for any $\gamma$ we have $|\chi^\pm(\gamma)|\le 3$ with equality if and only if $\gamma$ acts trivially on $\Lambda^\pm$.
The assertion follows.
\end{proof}

We can now finish the proof of Theorem~\ref{thm:Rigidity3D}.

\begin{proof}[Completion of the proof of Theorem~\ref{thm:Rigidity3D}.]
If $m(\curl,2)=m(\curl,-2)=3$ then $\Gamma$ must act trivially on $\Lambda^+$ and on $\Lambda^-$.
The only elements of $\SOv$ doing that are $\gamma=1$ and $\gamma=-1$.
Hence, either $\Gamma$ is trivial and $M=S^3$ or $\Gamma=\{1,-1\}$ and $M=\mathbb{RP}^3$.
\end{proof}

\begin{corollary}\label{cor:SpecAsym}
The $\curl$-spectrum on $M=\Gamma\backslash S^3$ is symmetric about $0$ if and only if 
$$
\sum_{\gamma\in\Gamma\setminus\{1\}}\frac{\chi^+(\gamma)-\chi^-(\gamma)}{\det(1-z\gamma)} = 0.
$$
\end{corollary}

\begin{proof}
This follows directly from
\begin{equation*}
F_+(z)-F_-(z)
=
\frac{1}{|\Gamma|}\sum_{\gamma\in\Gamma}\frac{\chi^+(\gamma)-\chi^-(\gamma)}{\det(1-z\gamma)}
\end{equation*}
and $\chi^+(1)=\chi^-(1)=3$.
\end{proof}

\begin{example}
Put
$$
R(\theta_1,\theta_2) :=
\begin{pmatrix}
\cos(\theta_1) & -\sin(\theta_1) & 0 & 0 \\
\sin(\theta_1) & \cos(\theta_1) & 0 & 0 \\
0 & 0 & \cos(\theta_2) & -\sin(\theta_2) \\
0 & 0 & \sin(\theta_2) & \cos(\theta_2)
\end{pmatrix}
\in \SOv \, .
$$
We choose $\Gamma=\{1,R(\frac{2\pi}{3},\frac{2\pi}{3}),R(\frac{4\pi}{3},\frac{4\pi}{3})\}$.
Then, $M=\Gamma\backslash S^3$ is called a \emph{lens space}.
If $e_1,e_2,e_3,e_4$ is a positively oriented orthonormal basis of $\R^4$ then $e_1\wedge e_2+e_3\wedge e_4, e_1\wedge e_4 + e_2\wedge e_3, e_1\wedge e_3 - e_2 \wedge e_4$ is a basis of $\Lambda^+$.
A straighforward computation shows that w.r.t.\ this basis the action of $R(\theta_1,\theta_2)$ on $\Lambda^+$ is given by the matrix
$$
\begin{pmatrix}
1 & 0 & 0 \\
0 & \cos(\theta_1+\theta_2) & \sin(\theta_1+\theta_2) \\
0 & -\sin(\theta_1+\theta_2) & \cos(\theta_1+\theta_2)
\end{pmatrix}
$$
and hence 
\begin{equation*}
\chi^+(R(\theta_1,\theta_2))=1+2\cos(\theta_1+\theta_2).
\end{equation*}
Similarly one sees 
\begin{equation*}
\chi^-(R(\theta_1,\theta_2))=1+2\cos(\theta_1-\theta_2).
\end{equation*}
In order to apply the criterion in Corollary~\ref{cor:SpecAsym} we compute 
\begin{align}
\chi^+(R({2\pi}/{3},{2\pi}/{3})) &= 1+2\cos({4\pi}/{3})=0, \label{eq:chi+1}
\\
\chi^-(R({2\pi}/{3},{2\pi}/{3})) &= 1+2\cos(0)=3, \label{eq:chi-1}
\\
\det(1-zR({2\pi}/{3},{2\pi}/{3}) 
&=
\det
\begin{pmatrix}
1-z\cos(2\pi/3) & z\sin(2\pi/3) \\
-z\sin(2\pi/3) & 1-z\cos(2\pi/3)
\end{pmatrix}
^2 \notag\\
&=
\det
\begin{pmatrix}
1+z/2 & z\sqrt{3}/2 \\
-z\sqrt{3}/2 & 1+z/2
\end{pmatrix}
^2\notag\\
&=
(1+z+z^2)^2\notag
\end{align}
Thus for $\gamma=R({2\pi}/{3},{2\pi}/{3})$ we get
$$
\frac{\chi^+(\gamma)-\chi^-(\gamma)}{\det(1-z\gamma)} 
=
\frac{-3}{(1+z+z^2)^2}.
$$
Similarly, for $\gamma=R({4\pi}/{3},{4\pi}/{3})$ we get
$$
\frac{\chi^+(\gamma)-\chi^-(\gamma)}{\det(1-z\gamma)} 
=
\frac{-3}{(1+z+z^2)^2}
$$
as well.
Corollary~\ref{cor:SpecAsym} now shows that the $\curl$-spectrum of the lens space $M$ is not symmetric about $0$.
Specifically, from \eqref{eq:chi+1}, \eqref{eq:chi-1}, the corresponding values for $\frac{4\pi}{3}$, and Corollary~\ref{cor:SphericalKleinst} we see that $m(\curl,2)=0$ while $m(\curl,-2)=3$.
\end{example}

It remains to prove Theorem~\ref{thm:SpecSpaceform}.
We denote the eigenspace of an operator $D$ on $S^3$ to the eigenvalue $\lambda$ by $\E(D,\lambda)$.
Let $\iota:S^3\hookrightarrow\R^4$ be the inclusion map.
We regard the elements of $\Lambda^\pm$ as constant (parallel) $2$-forms on $\R^4$.

\begin{lemma}\label{lem:Lambda+-}
The map $\omega\mapsto *\iota^*\omega$ yields $\SOv$-equivariant isomorphisms $\Lambda^+\to\E(\curl,2)$ and $\Lambda^-\to\E(\curl,-2)$.
\end{lemma}

\begin{proof}
Denote the exterior unit normal vector field of $S^3$ by $\n$ and the Levi-Civita connection of $\R^4$ by $\nb$.
For vector fields on $S^3$ the Gauss equation says
$$
\nb_XY = \nabla_XY -\<X,Y\>\n .
$$
This implies for $2$-forms $\omega$ on $\R^4$ (assuming without loss of generality $\nabla_XY=\nabla_XZ=0$ at the point under consideration):
\begin{align*}
(\nabla_X (\iota^*\omega))(Y,Z)
&=
\partial_X(\iota^*\omega(Y,Z)) \\
&=
\partial_X(\omega(\iota_*Y,\iota_*Z)) \\
&=
(\nb_X\omega)(\iota_*Y,\iota_*Z) + \omega(\nb_X(\iota_*Y),\iota_*Z)  + \omega(\iota_*Y,\nb_X(\iota_*Z)) \\
&=
(\iota^*\nb_X\omega)(Y,Z) - \omega(\<X,Y\>\n,\iota_*Z)  - \omega(\iota_*Y,\<X,Z\>\n) \\
&=
(\iota^*\nb_X\omega)(Y,Z) + \<X,Y\>\omega(\iota_*Z,\n)  - \<X,Z\>\omega(\iota_*Y,\n) .
\end{align*}
In particular, if $\omega$ is parallel then
\begin{equation}
(\nabla_X (\iota^*\omega))(Y,Z) = \<X,Y\>\omega(\iota_*Z,\n)  - \<X,Z\>\omega(\iota_*Y,\n) .
\label{eq:Gauss}
\end{equation}
Denote the Hodge-star operator on $S^3$ by $*$ and that on $\R^4$ by $\sb$.
Let $e_1$, $e_2$, $e_3$ be a local positively oriented orthonormal frame on $S^3$.
Then, $\n$, $e_1$, $e_2$, $e_3$ forms a positively oriented orthonormal frame on $\R^4$.
Using \eqref{eq:Gauss} we compute for parallel $\omega$:
\begin{align*}
\big(\nabla_X(*\iota^*\omega)\big)(e_1)
&=
*\big(\nabla_X(\iota^*\omega)\big)(e_1) \\
&=
\big(\nabla_X(\iota^*\omega)\big)(e_2,e_3) \\
&=
\<X,e_2\>\omega(\iota_*e_3,\n)  - \<X,e_3\>\omega(\iota_*e_2,\n) \\
&=
-\<X,e_2\>\sb\omega(\iota_*e_1,\iota_*e_2)  - \<X,e_3\>\sb\omega(\iota_*e_1,\iota_*e_3) \\
&=
-\sb\omega(\iota_*e_1,\iota_*X) \\
&=
\iota^*(\sb\omega)(X,e_1) .
\end{align*}
Since the direction of $e_1$ can be chosen arbitrarily we find
\begin{equation}
\nabla_X(*\iota^*\omega) = \iota^*(\sb\omega)(X,\cdot) .
\end{equation}
Now let $\omega\in\Lambda^+$.
Then, $\omega$ is $\nb$-parallel and satisfies $\omega=\sb\omega$.
Thus $\eta:=*\iota^*\omega\in\Omega^1_{C^\infty}(S^3)$ satisfies
$$
\nabla_X\eta
=
\iota^*\omega(X,\cdot)
=
(*\eta)(X,\cdot)
=
- *(X^\flat\wedge\eta) ,
$$
hence $\eta$ is $\nh$-parallel.
Since the space of $\nh$-parallel $1$-forms on $S^3$ coincides with $\E(\curl,2)$ the map $\omega\mapsto *\iota^*\omega$ restricts to a linear map $\Lambda^+\to \E(\curl,2)$.
The elements $\gamma\in\SOv$ act by orientation-preserving isometries, hence they commute with $*$ and $\iota^*$.
Thus the map is $\SOv$-equivariant.
Furthermore, the map is nontrivial, $\Lambda^+$ is an irreducible $\SOv$-module and $\Lambda^+$ and $\E(\curl,2)$ both have dimension $3$.
Thus by Schur's lemma the map is an isomorphism.

The statement about $\Lambda^-\to\E(\curl,-2)$ is analogous.
\end{proof}

\begin{proof}[Proof of Theorem~\ref{thm:SpecSpaceform}.]
Since the cotangent bundle is trivialized by $\nh$-parallel $1$-forms, the Hilbert space $\Omega^1_{L^2}(S^3)$ is spanned by products $f\eta$ where $f\in\Omega^0_{L^2}(S^3)$ and $\eta$ is $\nh$-parallel.
It is well known that $\Omega^0_{L^2}(S^3)=\bigoplus_{k=0}^\infty\H^k$ where $\H^k$ is the space of harmonic homogeneous polynomials of degree $k$ on $\R^4$, restricted to $S^3$, see e.g.\ \cite[Sec.~C.I.C]{BGM}.
Denote the character of the representation of $\SOv$ on $\H^k$ by $\chi_k$.
Then, by Lemma~\ref{lem:Lambda+-}, the character of the $\SOv$-module $\Omega^1_{L^2}(S^3)$ is given by $\sum_{k=0}^\infty\chi_k\cdot\chi^+$.
Since $b_1(S^3)=0$ the Hodge decomposition reads
$$
\Omega^1_{L^2}(S^3) = \left( d^\dagger\Omega^2_{\mathcal{D}'}(S^3) \oplus d\Omega^0_{\mathcal{D}'}(S^3)\right) \cap \Omega_{L^2}^1(S^3).
$$
Now $d:\Omega^0_{\mathcal{D}'}(S^3)\to \Omega^1_{\mathcal{D}'}(S^3)$ is $\SOv$-equivariant and has kernel $\H^0$.
Thus the character of the $\SOv$-module $d^\dagger\Omega^2_{\mathcal{D}'}(S^3) \cap \Omega_{L^2}^1(S^3)$ is given by $1+\sum_{k=0}^\infty\chi_k\cdot(\chi^+-1)$.

On $\H^k\otimes \E(\curl,2)\subset \Omega^1_{L^2}(S^3)$ the connection-Laplacian $\nh^*\nh$ acts as $\Delta\otimes \id$ and hence has the same eigenvalue as the Laplacian $\Delta$ on $\H^k$, namely $k(k+2)$, see \cite[Prop.~III.C.I.1]{BGM}.
We encode the multiplicities of the eigenvalues of $\nh^*\nh$ on $M=\Gamma\backslash S^3$, restricted to coexact forms, in the Poincar\'e series
\begin{align*}
G_+(z) 
&= 
1+ \sum_{k=0}^\infty \frac{1}{|\Gamma|}\sum_{\gamma\in\Gamma}\chi_k(\gamma)\cdot(\chi^+(\gamma)-1)z^k \\
&=
1+ \frac{1}{|\Gamma|}\sum_{\gamma\in\Gamma}\sum_{k=0}^\infty \chi_k(\gamma)\cdot(\chi^+(\gamma)-1)z^k.
\end{align*}
Similarly, we consider
$$
G_-(z) = 1+ \frac{1}{|\Gamma|}\sum_{\gamma\in\Gamma}\sum_{k=0}^\infty \chi_k(\gamma)\cdot(\chi^-(\gamma)-1)z^k.
$$
Ikeda has shown \cite[p.~81]{Ike} that
$$
\sum_{k=0}^\infty \chi_k(\gamma)z^k = \frac{1-z^2}{\det(1-z\gamma)}.
$$
Thus
$$
G_\pm(z) = 1 + \frac{1-z^2}{|\Gamma|}\sum_{\gamma\in\Gamma}\frac{\chi^\pm(\gamma)-1}{\det(1-z\gamma)} .
$$
From \eqref{eq:Weitzen} we get on $d^\dagger\Omega_{C^\infty}^2(S^3)$ (using $\kappa=1$, $\Ric=2$ and $\Delta=\curl^2$):
$$
(\curl-2)\curl = \nh^*\nh
$$
and hence
$$
(\curl-1)^2 = \nh^*\nh +1 .
$$
Now we observe
\begin{align*}
\E(\curl,2+k) \oplus \E(\curl,-k)
&=
\E(\curl-1,1+k) \oplus \E(\curl-1,-k-1) \\
&=
\E((\curl-1)^2,(1+k)^2) \\
&=
\E((\nh^*\nh +1)|_{d^\dagger\Omega_{C^\infty}^2},(1+k)^2) \\
&=
\E(\nh^*\nh|_{d^\dagger\Omega_{C^\infty}^2},(1+k)^2-1) \\
&=
\E(\nh^*\nh|_{d^\dagger\Omega_{C^\infty}^2},k(k+2)) .
\end{align*}
For our Poincar\'e series this means
\begin{equation*}
F_+(z) + z^2F_-(z)
=
G_+(z).
\end{equation*}
Similarly, we get
\begin{equation*}
F_-(z) + z^2F_+(z)
=
G_-(z).
\end{equation*}
Solving for $F_+$ we find
\begin{align*}
F_+(z)
&=
\frac{1}{1-z^4}(G_+(z)-z^2 G_-(z)) \\
&=
\frac{1}{1+z^2}\Big(1+\frac{1}{|\Gamma|}\sum_{\gamma\in\Gamma} \frac{\chi^+(\gamma)-1 -z^2(\chi^-(\gamma)-1)}{\det(1-z\gamma)}\Big)
\end{align*}
and similarly for $F_-$.
\end{proof}

\section*{Acknowledgments}

I would like to thank Dmitri Vassiliev for directing my attention to the $\curl$-operator on manifolds and Nikolai Saveliev for pointing out interesting references.

\end{document}